\documentclass[leqno]{siamart1116}

\usepackage{etoolbox}
\newtoggle{siam}

\toggletrue{siam}

\iftoggle{siam}{
}
{

}

\usepackage[utf8]{inputenc}
\usepackage{graphicx}
\usepackage{amsmath}
\usepackage{amssymb}
\usepackage{mathtools}
\usepackage{enumerate}
\usepackage{color}
\usepackage{epsfig}
\usepackage{hyperref}

\usepackage[ruled,vlined,linesnumbered,algo2e]{algorithm2e}

\newcommand{\M}{\mathcal{M}}
\newcommand{\TM}{\mathcal{T}_Y\mathcal{M}}
\newcommand{\T}{\mathcal{T}}
\newcommand{\bigO}{\mathcal{O}}
\newcommand{\R}{\mathbb{R}}
\renewcommand{\S}{{\bf S}}
\renewcommand{\L}{{\bf L}}
\newcommand{\bfr}{{\bf r}}

\DeclareMathOperator{\U}{\textbf{U}}
\DeclareMathOperator{\V}{\textbf{V}}
\DeclareMathOperator{\Y}{\textbf{Y}}
\DeclareMathOperator{\C}{\textbf{C}}
\DeclareMathOperator{\Q}{\textbf{Q}}
\DeclareMathOperator{\I}{\textbf{I}}
\DeclareMathOperator{\K}{\textbf{K}}

\DeclareMathOperator{\X}{\textbf{X}}
\DeclareMathOperator{\A}{\textbf{A}}
\DeclareMathOperator{\W}{\textbf{W}}
\DeclareMathOperator{\ten}{Ten}
\DeclareMathOperator{\mat}{\bf{Mat}}
\DeclareMathOperator{\bigtimes}{{\hbox{\large\sf X}}}
\DeclareMathOperator{\rank}{rank}
\DeclareMathOperator*{\Bigotimes}{\text{\raisebox{0.25ex}{\scalebox{0.8}{$\bigotimes$}}}}

\usepackage{accents}
\newcommand*{\dt}[1]{%
  \accentset{\mbox{\large\bfseries .}}{#1}}

\newcommand{\dr}{\mathrm{d}}

\iftoggle{siam}{
}
{
\usepackage{amsthm}

\newtheorem{theorem}{Theorem}

}

\allowdisplaybreaks

\iftoggle{siam}{
	\bibliographystyle{siam}
}
{
	\bibliographystyle{abbrv}
}

\title{Time integration of \\ rank-constrained Tucker tensors\footnote{\textnormal{September 8, 2017}}}
\author{Christian Lubich\thanks{Mathematisches Institut, Universit{\"a}t T{\"u}bingen, Auf der Morgenstelle 10, D--72076 T{\"u}bingen, Germany (lubich@na.uni-tuebingen.de)} \and Bart Vandereycken\thanks{Section of Mathematics, University of Geneva, Rue du Lièvre 2-4, 1211 Geneva, Switzerland (Bart.Vandereycken@unige.ch)} \and Hanna Walach\thanks{Mathematisches Institut, Universit{\"a}t T{\"u}bingen, Auf der Morgenstelle 10, D--72076 T{\"u}bingen, Germany (walach@na.uni-tuebingen.de).}}
\date{\today} 

\begin{document}

\maketitle

\iftoggle{siam}{
	\slugger{sinum}{xxxx}{xx}{x}{x--x}
}{}

\begin{abstract}
	Dynamical low-rank approximation in the Tucker tensor format of given large time-dependent tensors and of tensor differential equations is the subject of this paper. In particular, a discrete time integration method for rank-constrained Tucker tensors is presented and analyzed. It extends the known projector-splitting integrator for dynamical low-rank approximation of matrices to Tucker tensors and is shown to inherit the same favorable properties. The integrator is based on iteratively applying the matrix projector-splitting integrator to tensor unfoldings but with inexact solution in a substep. It has the property that it reconstructs time-dependent Tucker tensors of the given rank exactly. The integrator is also shown to be robust to the presence of small singular values in the tensor unfoldings. 
	Numerical examples with problems from quantum dynamics and tensor optimization methods illustrate our theoretical results.
\end{abstract}

\iftoggle{siam}{
	\pagestyle{myheadings}
	\thispagestyle{plain}
	\markboth{CH.~LUBICH, B.~VANDEREYCKEN AND H.~WALACH}{TIME INTEGRATION OF RANK-CONSTRAINED TUCKER TENSORS}
}{}

\section{Introduction}
In this paper we propose and study a discrete method for approximating time-dependent tensors $A(t)\in\mathbb{R}^{n_1\times\dots\times n_d}$ with $t_0\le t \le T$ by tensors of a prescribed (low) multilinear rank. The tensors $A(t)$ are either given explicitly or are the unknown solution to a tensor differential equation
\begin{align}\label{ode}
 \dt{A}(t) = F(t,A(t)), \qquad A(t_0) = A^0,
\end{align}
where $\dt A(t) = \dr A/\dr t$. The approximation follows the setting of the dynamical low-rank approximation of \cite{KL10}, which yields a differential equation for the approximation $Y(t)$ to $A(t)$ on the manifold $\M$ of tensors of multilinear rank $\bfr=(r_1,\dots,r_d)$. As is known from \cite{LMV00}, such tensors can be represented element-wise in the Tucker format (see, e.g., \cite{KB09}) as follows:
\begin{align}\label{eq:Tucker_elementwise}
 y_{k_1,\dots,k_d}(t) = \sum_{\ell_1=1}^{r_1}\cdots \sum_{\ell_d=1}^{r_d} c_{\ell_1,\dots,\ell_d}(t) \, u_{k_1,\ell_1}^{(1)}(t)
\cdots u_{k_d,\ell_d}^{(d)}(t),
\end{align}
with $k_i = 1, \ldots, n_i$ for all modes $i=1,\ldots, d$. Using the multilinear product $\bigtimes$ (see, e.g.,~\cite{KB09}), the relation~\eqref{eq:Tucker_elementwise} can be written more succinctly as
\[
Y(t) = C(t) \underset{i=1}{\overset{d}{\bigtimes}} \U_i(t),
\]
where $C(t) \in \R^{r_1 \times \dots \times r_d}$ is the time-dependent core tensor of full multilinear rank with entries $c_{\ell_1,\dots,\ell_d}(t)$, and $\U_i(t)$ is the mode-$i$ time-dependent basis matrix of size $n_i \times r_i$ with entries $u_{k_i,\ell_i}^{(i)}(t)$. 

The differential equation for $Y(t)\in\M$ is obtained by projecting $F(t,Y(t))$ (or $\dt A(t)$ in the case of a given explicit time-dependent tensor $A(t)$) onto the tangent space $\T_{Y(t)}\M$ of $\M$ at the current approximation $Y(t)\in\M$:
\begin{align}\label{dlra}
 \dt{Y}(t) = P(Y(t))F(t,Y(t)), \qquad Y(t_0) = Y^0 \in \M,
\end{align}
where $P(Y)\colon \R^{n_1\times\dots\times n_d} \to \T_{Y}\M$ is the orthogonal projection, which can be given explicitly as an alternating sum of subprojections \cite{KL10,L15}. 
The differential equation \eqref{dlra} needs to be solved numerically in virtually all applications. For example, in the context of molecular quantum dynamics, such an approach is taken in the multiconfiguration time-dependent Hartree method (MCTDH) \cite{MGW09}, where the multivariate wavefunction is approximated by a linear combination of products of univariate functions.

For the corresponding matrix problem (i.e., the particular case $d=2$), a projector-splitting integrator with remarkable properties has been proposed in~\cite{LO14}. In particular, contrary to standard numerical integrators such as explicit or implicit Runge--Kutta methods, that integrator is robust to the presence of small singular values of the current approximation matrix; see~\cite[Thm.~2.1]{KLW16} (which is restated as Theorem~\ref{thm:robustMatrix} below). Such a situation arises when the rank is chosen sufficiently large as to obtain an accurate approximation and the singular values have a decaying behaviour. This is, e.g., typical in applications where the matrix comes from discretising a smooth bivariate function or it is the solution of a parametric or stochastic PDE; see, e.g.,~\cite{Hackbusch:2012,Kressner:2015b}.

We outline the contributions and organization of the paper as follows:

The matrix projector-splitting integrator has been extended to Tucker tensors in \cite{L15} and to tensor trains (or matrix product states in the terminology of physics) in \cite{LOV15,HLOVV16}. In this paper we give a conceptually different derivation of an integrator for \eqref{dlra}, based on the idea of an inexact solution of substeps within the matrix projector-splitting integrator applied to matricizations of \eqref{dlra}. This derivation allows us to transfer the known favorable properties of the matrix integrator to the tensor case. We then show that the newly derived integrator is mathematically equivalent to the tensor projector-splitting integrator of \cite{L15}, whose key properties of exactness and robustness are thus proven in the present paper. We mention that this integrator has meanwhile proved its robustness and efficiency in a first MCTDH implementation \cite{KBL17}.

In Section 2, we briefly restate the matrix projector-splitting integrator and some of its properties. We then derive in Section 3 the Tucker tensor integrator in a recursive way from the matrix integrator. In Section 4, we show that the integrator reproduces the given matrix $A(t)$ if it is explicitly given and it is of rank $(r_1,\dots,r_d)$. This extends the exactness property of the integrator in the matrix case \cite{LO14}, which is fundamental for the error analysis of the matrix projector-splitting integrator in \cite{KLW16}. In Section 5, we extend the error analysis of \cite{KLW16} to the Tucker tensor case, which shows the robustness of the integrator in the presence of small singular values of its matricizations. In Section 6, we show that the integrator derived and studied here is mathematically equivalent to the projector-splitting integrator for Tucker tensors as proposed in \cite{L15}. Finally, in Section 7, we present numerical experiments that illustrate the behaviour of the integrator.

\section{The matrix projector-splitting integrator}
In this section, we briefly restate the projector-splitting integrator from~\cite{LO14} for the matrix case, i.e., for  $d=2$. Recall that our aim is to numerically integrate the initial value problem~\eqref{dlra} to obtain a low rank approximation of~\eqref{ode}. To this end, we will make use of the SVD-like representation 
\begin{align*}
 \Y(t) = \U(t)\S(t)\V(t)^T
\end{align*}
of the rank-$r$ approximation matrix $\Y(t) \in \M$, where $\U(t)$ and $\V(t)$ are basis matrices of size $n_1 \times r$ and $n_2 \times r$ for the first and second mode, respectively. The invertible matrix $\S(t) \in \R^{r \times r}$ has the same nonzero singular values as $\Y(t)$,  but unlike the SVD, $\S(t)$ is not assumed to be diagonal. 

The projector-splitting integrator updates the factors $\U,\S,\V$, starting from the initial value $\Y^0 = \U^0\S^0\V^{0,T}$. Let $F:\R \times \R^{n_1 \times n_2} \to \R^{n_1 \times n_2}$. Then one time step from $t_0$ to $t_1 = t_0+h$ proceeds as follows:
\begin{enumerate}
 \item \textbf{K-step}: Update $\U^0 \to \U^1$, $\S^0 \to \widehat{\S}^1$. \\
 Integrate to $t=t_1$ the differential equation 
 \begin{align}\label{eq:K-step}
  \dt{\K}(t) = F(t,\K(t)\V^{0,T})\V^0, \qquad \K(t_0) = \U^0\S^0
 \end{align}
and perform a QR factorization $\K(t_1) = \U^1\S^1$ to orthonormalise the columns of $\K(t_1)$. This yields $\U^1$ as the final approximation of the basis matrix $\U(t)$ at $t=t_1$, and the temporary update $\widehat{\S}^1$.
 \item \textbf{S-step}: Update $\widehat{\S}^1 \to \widetilde{\S}^0$. \\
 Integrate to $t=t_1$ the differential equation 
 \begin{align}\label{eq:S-step}
  \dt{\S}(t) = -\U^{1,T} F(t,\U^1\S(t)\V^{0,T})\V^0, \qquad \S(t_0) = \widehat{\S}^1
 \end{align}  
 This yields the temporary update $\widetilde{\S}^0 = \S(t_1)$.
 \item \textbf{L-step}: Update $\V^0 \to \V^1$, $\widetilde{\S}^0 \to {\S}^1$.  \\
 Integrate to $t=t_1$ the differential equation 
 \begin{align}\label{eq:L-step}
  \dt{\L}^T(t) = \U^{1,T} F(t,\U^1\L(t)^T), \qquad \L^T(t_0) = \widetilde{\S}^0\V^{0,T}
 \end{align}
and perform a QR factorization $\L(t_1) = \V^1\S^{1,T}$. This yields the final approximations $\V^1$ and $\S^{1}$.
 \end{enumerate}
Merging the computed factors results in the rank-$r$ approximation matrix 
\begin{align}
 \Y^1 = \U^1 \S^1 \V^{1,T}
\end{align}
after one time step. To continue in time, we take the factorized matrix $\Y^1 $ as initial value for the next time step and apply this  scheme again. This way we obtain a first-order splitting method for~\eqref{dlra}.

The matrix projector-splitting integrator has a remarkable exactness property. 

\begin{theorem}{\rm \cite[Thm.~4.1]{LO14}}\label{matex}
 Let $\A(t) \in \mathbb{R}^{n_1 \times n_2}$ with $\rank \A(t) \leq r$ for all $t$ and $\A(t_0) = \Y^0$. Further, let $\V(t_1)^T \V(t_0)$ be invertible. Then, the splitting integrator described above (with $F(t,\Y) = \dt{\A}(t)$) reproduces the exact solution: $\Y^1 = \A(t_1)$.
\end{theorem}

Moreover, the integrator is robust to the presence of small singular values in the solution or its approximation. 
 
 \begin{theorem}{\rm \cite[Thm.~2.1]{KLW16}}\label{thm:robustMatrix}
	 Let $\A(t)$ denote the solution of~\eqref{ode} on $[t_0,T]$ in case of $d=2$ and let $\M$ be the manifold of rank $r$ matrices in $\R^{n_1 \times n_2}$. Suppose the following assumptions are satisfied with $\|\cdot\|$ the Euclidean norm.
	  \begin{enumerate}[(a)]
	  \item\label{Lip-Mat} $F(t,\Y)$ is Lipschitz continuous and bounded for all $\Y, \widetilde{\Y} \in \R^{n_1 \times n_2}$:
	  \begin{align*}
	   \| F(t,\Y) - F(t,\widetilde{\Y}) \| \leq L \| \Y - \widetilde{\Y} \|, \qquad
	   \| F(t,\Y) \| \leq B.
	  \end{align*}
	  \item\label{MR-Mat} $F(t,\Y)$ can be decomposed into a tangential part and a small perturbation: 
	  \begin{align*}
	   &F(t,\Y) = M(t,\Y) + R(t,\Y), \\
	   &M(t,\Y) \in \T_{\Y}\M, \quad \|R(t,\Y)\| \leq \varepsilon,
	  \end{align*}
	  for all $\Y\in \M$ in a neighborhood of $\A(t)$ and for all $t \in [t_0,T]$. 
	 \item\label{A0-Mat} The initial value $\A(t_0)$ for \eqref{ode} has  rank $r$.
	  \end{enumerate}
 Then, the error of the splitting integrator described above after $n$ steps with step size $h>0$ satisfies for all $t_n = t_0 + nh \leq T$
\begin{align*}
 \| \Y_n - \A(t_n) \| \leq c_1 h + c_2 \varepsilon,
\end{align*}
where the constants $c_1, c_2$ only depend on $L,B,T-t_0$. In particular, the constants are independent of singular values of the exact or approximate solution matrix.
\end{theorem}

In general, the differential equations in the substeps~\eqref{eq:K-step}--\eqref{eq:L-step} have to be solved numerically, e.g., by a Runge--Kutta method. In the case when $F(t,\Y) = \dt{\A}(t)$ for explicitly given matrices $\A(t)$, the integrator works with the increment $\A(t_1) - \A(t_0)$ and so the substeps can be solved directly. If, however, we apply a numerical integrator, then instead of $\Y_n$, we compute a perturbed matrix $\widetilde{\Y}_n$. Assuming that the arising local errors in the substeps are bounded by $h\eta$, the error bound of \eqref{thm:robustMatrix} after $n$ time steps is given by 
\begin{equation}
\label{err-inexact}
\| \widetilde{\Y}_n - \A(t_n) \| \leq c_1 h + c_2 \varepsilon + c_3 \eta,
\end{equation} 
where $c_3$ also only depends on $L, B, T-t_0$; see Section 2.6.3 in \cite{KLW16}.

 These exactness and robustness properties will be extended to the Tucker integrator in Sections~\ref{exactness} and~\ref{sec:error_robustness}, respectively. But first, we present the integration scheme for tensors in the Tucker format.
 
\section{The nested Tucker integrator}\label{NTucker}

\subsection{Derivation of the integrator}\label{sec:NTucker-derivation}
To find a low-rank approximation for~\eqref{ode} in the case of Tucker tensors of general dimension $d$, we will now extend  the matrix projector-splitting integrator to tensors. To this end, we will need to transfer tensors into a matrix setting by considering their matricizations. In particular, we denote by
\begin{align*}
 \mat_i(X) = \X \in \R^{n_i \times n_{1} \cdots n_{i-1} n_{i+1} \cdots n_d}
\end{align*}
the $i$-mode matricization of a tensor $X \in \R^{n_1 \times \cdots \times n_d}$. It arranges the mode-$i$ fibres of $X$ to be the rows of the resulting matrix $\X$. The reversal of the $i$-mode matricization is called tensorization, which we denote as 
\begin{align*}
 \ten_i(\X) = X \in \R^{n_1 \times \cdots \times n_d}.
\end{align*}

We begin by matricizing the \emph{tensor} ODE~\eqref{ode} in mode 1:
\begin{align}\label{eq:NTucker_Mat_ODE_1}
 \mat_1\bigl(\dt{A}(t)\bigr) = \mat_1\bigl(F(t,A(t))\bigr), \qquad \mat_1\bigl(A(t_0)\bigr) = \mat_1\bigl(A^0\bigr).
\end{align}
This will allow us to formally apply the matrix projector-splitting integrator to this \emph{matrix} ODE where the initial value $\mat_1\bigl(A^0\bigr)$ is approximated by $\mat_1\bigl(Y^0\bigr)$ with $Y^0 \in \M$. Since $Y^0$ has multilinear rank $(r_1,\ldots,r_d)$, it satisfies the decomposition
\[
Y^0 = C^0 \underset{i=1}{\overset{d}{\bigtimes}} \U_i^0
\]
with $C^0 \in \R^{r_1 \times \dots \times r_d}$ and $\U_i^0 \in \R^{n_i \times  r_i}$.  Since we are dealing with the $1$st step of the algorithm, let us denote the initial value as $Y_1^0:=Y^0$ with core tensor as $C_1^0 := C^0$. By performing the QR decomposition 
$$\mat_1(C_1^0)^T = \Q_1^{0} \S_1^{0,T} \in \R^{r_2 \cdots r_d \times r_1}$$
and denoting 
\begin{equation}\label{eq:def_V_1}
	\V_1^{0,T} = \Q_1^{0,T} \Bigotimes_{i=2}^d \U_i^{0,T} \in \R^{r_1 \times n_2 \cdots  n_d},
\end{equation}	
we obtain the necessary SVD-like representation of the initial value of~\eqref{eq:NTucker_Mat_ODE_1} as
\begin{align*}
 \mat_1(Y_1^0) = \U_1^0 \mat_1(C_1^0) \underset{i=2}{\overset{d}{\Bigotimes}} \U_i^{0,T} = \U_1^0 \S_1^0 \V_1^{0,T}.
\end{align*}

Now, we are in the situation to apply the matrix projector-splitting integrator to~\eqref{eq:NTucker_Mat_ODE_1}:
\begin{enumerate}
 \item \textbf{K-step}: Update $\U_1^0 \to \U_1^1$, $\S_1^0 \to \widehat{\S}_1^1$.
 \item \textbf{S-step}: Update $\widehat{\S}_1^1 \to \widetilde{\S}_1^0$.
 \item \textbf{L-step}: Update $\V^0_1 \to \V^1_1$, $\widetilde{\S}^0_1 \to {\S}^1_1$ by  solving approximately
 \begin{equation}\label{odeL}
 \begin{aligned}
  \dt{\L}_1^T(t) &= \U_1^{1,T} \mat_1\left(F(t,\ten_1(\U_1^1\L_1^T(t)))\right), \\ \L_1^T(t_0) &= \L_1^{0,T} =  \widetilde{\S}_1^0\V_1^{0,T},
 \end{aligned}
 \end{equation}
 with $\L_1 \in \R^{r_1 \times n_2 \cdots n_d}$ and the QR factorization $\L_1(t_1) = \V^1_1\S_1^{1,T}$.
\end{enumerate}
The K- and S-steps can be calculated as in the matrix case, i.e., solving~\eqref{eq:K-step} and~\eqref{eq:S-step} but applied to~\eqref{eq:NTucker_Mat_ODE_1}. However, we do not solve the matrix differential equation in the L-step directly since it is defined for a prohibitively large $\L_1$. More importantly, it would also not lead to an approximation for $Y(t_1)$ of multilinear rank $(r_1, \ldots, r_d)$ since the exact L-step above only reduces the rank of the first mode. Instead, we perform a low-rank approximation  for~\eqref{odeL} by applying the matrix projector-splitting integrator again to a reshaped version of it.

Defining $Y_2(t) = \ten_1({\L_1^{T}}(t)) \in \R^{r_1 \times n_2 \times \cdots \times n_d}$, we first retensorize~\eqref{odeL} as
\begin{align}\label{odeY}
 \dt{Y}_2(t) = F(t,Y_2(t) \times_1 \U_1^1)\times_1 \U_1^{1,T}, \qquad Y_2(t_0) = \ten_1({\L_1^{0,T})}.
\end{align}
Observe that $Y_2(t)$ is usually of significantly smaller size than $Y(t)$ since typically $r_1 \ll n_1$. Next, we unfold~\eqref{odeY} in the $2${nd} mode. For simplicity of notation, we denote this $2$-mode unfolding of $Y_2$ by $\Y_{[2]} := \mat_2 (Y_{2}) \in \R^{n_2 \times r_1 n_{3} \cdots n_d}$. This gives the \emph{matrix} differential equation
\begin{align*}
\dt{\Y}_{[2]}(t) = \mat_2\Bigl(F(t,\ten_2(\Y_{[2]}(t)) \times_1 \U_1^1)\times_1 \U_1^{1,T}\Bigr)
\end{align*}
and, using~\eqref{odeL} and~\eqref{eq:def_V_1}, the initial value 
\[
\Y_{[2]}(t_0) = \mat_2\left(\ten_1({\L_1^{0,T})}\right)  = \mat_2\Bigl( \ten_1( \widetilde{\S}_1^0 \Q_1^{0,T}  \underset{i=2}{\overset{d}{\Bigotimes}} \U_i^{0,T}) \Bigr).
\]
Defining $C_2^0 = \ten_1\bigl(\widetilde{\S}_1^0 \Q_1^{0,T}\bigr) \in \R^{r_1 \times \cdots \times r_d}$ and $\C_{[2]}^0 = \mat_2(C_2^0)$, we also have
\begin{align*}
 \Y_{[2]}(t_0) = \mat_2 \bigl(C_2^0 \underset{i=2}{\overset{d}{\bigtimes}} \U_i^{0}  \bigr) = \U_2^0 \C^0_{[2]} \Bigl(\I_{r_1} \otimes \underset{i=3}{\overset{d}{\Bigotimes}} \U_i^{0,T} \Bigr).
\end{align*}
As before we have to determine the SVD-like representation of $\Y_{[2]}(t_0)$.  To this end, compute the QR factorization  $\C_{[2]}^{0,T} = \Q_2^{0}\S_2^{0,T}$. We then obtain the desired result as $\Y_{[2]}(t_0) = \U_2^0 \S_2^0 \V_2^{0,T}$ with $\V_2^{0,T} = \Q_2^{0,T} \Bigl(\I_{r_1} \otimes \Bigotimes_{i=3}^d \U_i^{0,T}\Bigr) \in \R^{r_2 \times r_1 n_3 \cdots n_d}$.

%
Now that we have set up the matrix problem again, we can apply the matrix projector-splitting integrator to $\dt{\Y}_{[2]}(t)$.
\begin{enumerate}
 \item \textbf{K-step}: Update $\U_2^0 \to \U_2^1$, $\S_2^0 \to \widehat{\S}_2^1$
 \item \textbf{S-step}: Update $\widehat{\S}_2^1 \to \widetilde{\S}_2^0$ 
 \item \textbf{L-step}: Update $\V^0_2 \to \V^1_2$, $\widetilde{\S}^0_2 \to {\S}^1_2$ by  solving approximately
 \begin{align*}
  \dt{\L}_2^T(t) &= \U_2^{1,T} \mat_2\Bigl(F(t,\ten_2(\U_2^1\L_2^T(t)) \times_1 \U_1^1)\times_1 \U_1^{1,T}\Bigr), \\
  \L_2^{T}(t_0) &= \L_2^{0,T} = \widetilde{\S}_2^0\V_2^{0,T},
 \end{align*}
 with $\L_2 \in \R^{r_2 \times r_1 n_3 \cdots n_d}$ and the QR factorization $\L_2(t_1) = \V^1_2\S_2^{1,T}$.
 \end{enumerate}
 
 
We continue recursively with solving the L-step approximately in each iteration step of the integrator. Generalising the pattern for modes 1 and 2 from above to general $i$, this requires us to find $Y_i(t) \in \R^{r_1 \times \cdots \times r_{i-1} \times n_{i} \times \cdots \times n_d}$ that satisfies the ODE
\begin{align}\label{odeYi}
 \dt{Y}_{i}(t) = F\Bigl(t,Y_{i}(t) \underset{k=1}{\overset{i-1}{\bigtimes}} \U_k^1 \Bigr)  \underset{k=1}{\overset{i-1}{\bigtimes}} \U_k^{1,T}, \qquad Y_{i}(t_0) = \ten_{i-1}\bigl(\L_{i-1}^{0,T}\bigr).
\end{align}
The K- and S-steps for mode $i-1$ calculate, in particular, the matrices $\widetilde{\S}_{i-1}^0$ and $\Q_{i-1}^{0}$. This implies that the initial guess in the above ODE is available as
\[
 \L_{i-1}^{0,T} = \widetilde{\S}_{i-1}^0\Q_{i-1}^{0,T} \Bigl(\Bigotimes_{k=1}^{i-2} \I_{r_k} \otimes \underset{k=i}{\overset{d}{\Bigotimes}} \U_k^{0,T}  \Bigr) \in \R^{n_{i-1} \times r_1 \cdots r_{i-2} n_{i} \cdots n_d}.
\]
To obtain a suitable matrix version of~\eqref{odeYi}, we unfold it in mode $i$ and define $\Y_{[i]} = \mat_i(Y_i) \in \R^{n_{i} \times r_1 \cdots r_{i-1} n_{i+1} \cdots n_d}$. This gives
\begin{equation}\label{Yi}
\begin{aligned}
 \dt{\Y}_{[i]}(t) &= \mat_i\Bigl(F\Bigl(t,\ten_i\bigl(\Y_{[i]}(t)\bigr) \underset{k=1}{\overset{i-1}{\bigtimes}} \U_k^1 \Bigr) \underset{k=1}{\overset{i-1}{\bigtimes}} \U_k^{1,T} \Bigr), \\
 \Y_{[i]}^0 
 &= \mat_i\Bigl(C_{i}^0 \underset{k=1}{\overset{i-1}{\bigtimes}} \I_{r_k} \underset{k=i}{\overset{d}{\bigtimes}} \U_{k}^0 \Bigr) = \U_{i}^0 \S_{i}^0 \V_{i}^{0,T}
\end{aligned}
\end{equation}
with $C_i^0 = \ten_{i-1}\bigl(\widetilde{\S}_{i-1}^0 \Q_{i-1}^{0,T}\bigr) \in \R^{r_1 \times \cdots \times r_d}$ and the QR decomposition
\[
 \mat_i(C_i^0)^T = \Q_i^{0} \S_i^{0,T} \in \R^{r_1 \cdots r_{i-1} r_{i+1} \cdots r_d \times r_i}.
\]
In addition, we have also used
\begin{equation}\label{eq:def_Vi}
	\V_{i}^{0,T} = \Q_{i}^{0,T}  \Bigl(\Bigotimes_{k=1}^{i-1} \I_{r_k} \otimes \Bigotimes_{k=i+1}^d \U_k^{0,T} \Bigr) \in \R^{r_i \times r_1 \cdots r_{i-1} n_{i+1} \cdots n_d}.
\end{equation}

In this way, we can indeed apply the K- and S-steps of the matrix projector- splitting to~\eqref{Yi}. The L-step is recursively solving
\begin{equation}\label{Li}
  \begin{aligned}
  \dot{\L}_i^T(t) &= \U_i^{1,T} \mat_i\Bigl(F\Bigl(t,\ten_i\bigl(\U_i^1\L_i^T(t)\bigr) \underset{k=1}{\overset{i-1}{\bigtimes}} \U_k^1 \Bigr)  \underset{k=1}{\overset{i-1}{\bigtimes}} \U_k^{1,T} \Bigr), \\
  \L_i^{T}(t_0) &= \L_i^{0,T} = \widetilde{\S}_i^0\V_i^{0,T}.
 \end{aligned}
 \end{equation}
with the scheme we just explained. The recursion ends at $i=d$ since then the L-step,
 \begin{equation}\label{LC}
  \begin{aligned}
   \dt{\L}_{d}^T &= \U_{d}^{1,T} \mat_d\Bigl(F\Bigl(t,\ten_d\bigl(\U_d^1 \L_d^T\bigr) \underset{i=1}{\overset{d-1}{\bigtimes}} \U_i^1 \Bigr)  \underset{i=1}{\overset{d-1}{\bigtimes}} \U_i^{1,T} \Bigr), \\
  \L_d^{T}(t_0) &= \L_d^{0,T} = \widetilde{\S}_d^0\V_d^{0,T},
 \end{aligned}
 \end{equation}
 can then be solved explicitly for $\L_d^T(t_1) \in \R^{r_d \times r_1 \cdots r_{d-1}}$. Observe that this means that $\L_d(t)$ actually corresponds to the update of the core tensor $C(t)$ itself. Hence, with such an explicit L step we have calculated the final update $\ten_d(\L_d(t_1)) =  C^1$. 
 
 The scheme from above operates on tensors $Y_i(t)$ that consecutively get smaller for $i=1,2,\ldots,d$. However, we can also interpret it as computing an approximation $Y^1$ for the Tucker tensor $Y(t_1) \in \R^{n_1 \times \cdots \times n_d}$ in~\eqref{dlra}. In particular, we have
 \[
  Y^1 = \ten_1(\U_1^1 \L_1^{T}(t_1)) = \ten_1\left(\U_1^1 \mat_1(Y_2(t_1))\right) = Y_2(t_1) \times_1 \U_1^1 \\
 \]
 with $\L_1(t_1)$ the approximate solution of~\eqref{odeL} obtained using $Y_2(t_1)$ in~\eqref{odeY}. In turn, $Y_2(t_1)$ is solved similarly using $Y_3(t)$:
 \[
  Y_2(t_1) = \ten_2(\U_2^1 \L_2^{T}(t_1)) =  Y_3(t_1) \times_2 \U_2^1.
 \]
 Hence, continuing recursively for all modes, we obtain
 \begin{align*}
  Y^1 &= Y_2(t_1) \times_1 \U_1^1 = Y_3(t_1) \times_2 \U_2^1 \times_1 \U_1^1  
  \\ &= \cdots = Y_{i+1}^1 \underset{k=1}{\overset{i}{\bigtimes}} \U_k^1 = C^1 \underset{k=1}{\overset{d}{\bigtimes}} \U_k^1.
 \end{align*}

\subsection{Practical algorithm}
As explained above, the nested Tucker integrator follows the scheme of recursively applying the matrix projector-splitting integrator with solving the first two steps, but performing a low-rank approximation for the third substep in each mode. The implementation of this integration scheme is straightforward and results in Alg.~\ref{alg:nestedKSL}. It simply updates the basis matrices $\U_i$ in the {K-step} and the auxiliary matrix $\S_i$ in the {S-step} for each mode. Quite remarkably, in the approximate L-step it suffices to only update the core tensor $C^0$. This also reduces the size of the matrix differential equation that has to be solved for the next mode.  For computational efficiency, we have written the operations using multilinear products. For example, line~\ref{alg:line:V_i}  is equivalent to~\eqref{eq:def_Vi}.


 The differential equations for $\K, \S, \L$ that need to be solved during the integration scheme, can be solved approximately, e.g., by a Runge--Kutta method. In the case, where $F(t,Y)$ is solution-independent and solely given by a tensor $A(t) \in \R^{n_1 \times \cdots \times n_d}$, those differential equations can be solved directly.  
 
 The integration scheme in Alg.~\ref{alg:nestedKSL} consists of recursively applying the matrix projector-splitting integrator. Since we do not solve the full matrix scheme, but rather the first two steps in order to update $\U_i$ and $\S_i$ for all modes $i = 1, \dots, d$, it is a  nested matrix projector-splitting integrator for Tucker tensors, or in short, the nested Tucker integrator.

\hfill \break

\begin{algorithm}[H]\label{alg:nestedKSL}
\caption{One time step of the nested Tucker integrator}
    \label{algntucker}
        \KwData{Tucker tensor $Y^0 = C^0 \bigtimes_{i=1}^d \U_i^0$, $F(t,Y)$, $t_0$, $t_1$}
        \KwResult{Tucker tensor $Y^1 = C^1 \bigtimes_{i=1}^d \U_i^1$}
    \Begin{
           \For {$i=1$ \KwTo $d$}
           {
             compute QR factorization $\mat_i(C^0)^T = \Q_i^0 \S_i^{0,T}$ 
            
            set $\V_i^{0,T} = \mat_i \Bigl(\ten_i(\Q_i^{0,T}) \underset{l=i+1}{\overset{d}{\bigtimes}} \U_l^{0,T}\Bigr)$  \label{alg:line:V_i} 
            
             set $\K_i^0 = \U_i^0 \S_i^0$
             
             set $\Y_{[i]}^+(t) = \K_i(t)\V_i^{0,T}$
            
             solve $\dt{\K}_i(t) = \mat_i\Bigl(F\bigl(t, \ten_i(\Y_{[i]}^+)\underset{k=1}{\overset{i-1}{\bigtimes}} \U_k^1\bigr) \underset{k=1}{\overset{i-1}{\bigtimes}} \U_k^{1,T}\Bigr) \V_i^0$, \\ \linebreak
             \ with initial value $\K_i(t_0) = \K_i^0$ and return $\K_i^1 = \K_i(t_1)$
            
             compute QR factorization $\K_i^1 = \U_i^1 \widehat{\S}_i^1$ \label{alg-line:QR_U1}

             set $\Y_{[i]}^-(t) = \U_i^1\S_i(t)\V_i^{0,T}$
            
             solve $\dt{\S}_i(t) = -\U_i^{1,T} \mat_i\Bigl(F\bigl(t, \ten_i(\Y_{[i]}^-)\underset{k=1}{\overset{i-1}{\bigtimes}} \U_k^1\bigr) \underset{k=1}{\overset{i-1}{\bigtimes}} \U_k^{1,T}\Bigr) \V_i^0$, \\ \linebreak
             \ with initial value $\S_i(t_0) = \widehat{\S}_i^1$ and return $\widetilde \S_i^0 = \S_i(t_1)$
            
             set $C^0 = \ten_i(\widetilde \S_i^0 \Q_i^{0,T})$
           }
        
           set $\L^{0,T}= \mat_d(C^0)$ 
        
           solve $\dt{\L}^T(t) = \U_d^{1,T} \mat_d\Bigl(F\bigl(t, \ten_d(\U_d^1\L(t)^T)\underset{k=1}{\overset{i-1}{\bigtimes}} \U_k^1\bigr)
           \underset{k=1}{\overset{d-1}{\bigtimes}} \U_k^{1,T}\Bigr)$, \\ \linebreak
           with initial value $\L^T(t_0) = \L^{0,T}$ and return $\L^{1,T} = \L^T(t_1)$
           
           set $C^1 = \ten_d(\L^{1,T})$
        
           set $Y^1 = C^1 \underset{i=1}{\overset{d}{\bigtimes}} \U_i^1$ 
         }
\end{algorithm}
\vspace{10pt}

\section{An exactness property of the integrator}\label{exactness}

Let $\M \subset \R^{n_1 \times \cdots \times n_d}$ be the manifold of tensors with multilinear rank $(r_1, \dots, r_d)$. Suppose that $A(t) \in \R^{n_1 \times \cdots \times n_d}$ is given explicitly, hence, we formally have $F(t,Y) = \dt{A}(t)$ in~\eqref{ode} and~\eqref{dlra}. In addition, we assume that $A(t) \in \M$ for $t_0 \leq t \leq T$. Our aim in this section is to prove that Alg.~\ref{algntucker}, the nested Tucker integrator, is in that case exact. In other words, Alg.~\ref{algntucker} solves the initial value problem~\eqref{dlra} exactly even though it is a discrete time stepping method. As mentioned in Theorem~\ref{matex}, the projector-splitting integrator for matrices already has this property but it does not hold for more standard integrators on $\M$, like the projected Runge--Kutta methods in~\cite{Kieri_V:2017}.


Since $A(t) \in \M$ for all $t$, we can write its $i$-mode matricization as
\begin{align} \label{USW}
 \mat_i(A(t)) = \U_i(t) \S_i(t) \W_i(t)^T,
\end{align}
where $\U_i(t) \in \R^{n_i \times r_i}$ and $\W_i(t) \in \R^{n_1 \cdots n_{i-1} \cdot n_{i+1} \cdots n_d \times r_i}$ have orthonormal columns and $\S_i(t) \in \R^{r_i \times r_i}$ for all $i = 1, \dots, d$. With this SVD-like representation we can state and prove the following exactness result. 

\begin{theorem}
 Let $A(t)$ be of multilinear rank $(r_1, \dots, r_d)$ for all $t \in (t_0,t_1)$ and let $Y(t_0) = A(t_0)$. Further, let $\W_i(t_1)^T \W_i(t_0)$ be invertible for all $i = 2, \dots, d$. Then, Algorithm \ref{algntucker} for $F(t,Y) = \dt{A}(t)$ reproduces the exact solution: $Y^1 = A(t_1)$.
\end{theorem}
\begin{proof}
Recall that the nested Tucker integrator in Alg.~\ref{algntucker} is designed to approximately solve the initial value subproblems (see~\eqref{Yi})
\begin{align}\label{expl}
  \dt{\Y}_{[i]}(t) = \mat_i\Bigl(\dt{A}(t) \underset{k=1}{\overset{i-1}{\bigtimes}} \U_k^{1,T} \Bigr), \qquad \Y_{[i]}(t_0) = \Y_{[i]}^0 = \U_{i}^0 \S_{i}^0 \V_{i}^{0,T}, 
\end{align}
where $\ten_i(\V_{i}^{0,T}) = \ten_i(\Q_{i}^{0,T}) \bigtimes_{l=i+1}^d \U_l^0 \in \R^{r_1 \times \cdots \times r_i \times n_{i+1} \times \cdots \times n_d}$ for each mode $i = 1, \dots, d$. In addition, the tensorized result $Y_i^1 = \ten_i(\Y_{[i]}(t_1))$ after one time step is in the low-rank manifold $\mathcal{M}_i := \{Y_i \in \R^{r_1 \times \cdots \times r_{i-1} \times n_i \times \cdots \times n_d}: \rank \mat_i(Y_i) = r_i \}$.

In the first part of the proof, we show that the initial value for \eqref{expl} can be written in terms of $A(t_0)$: 
\begin{align}\label{initYi}
 \Y_{[i]}(t_0) = \mat_i\Bigl(A(t_0) \underset{k=1}{\overset{i-1}{\bigtimes}} \U_k^{1,T}\Bigr).
\end{align}
This ensures that $\Y_{[i]}^0$ has rank $r_i$. With this initial value, the exact solution of \eqref{expl} has rank $r_i$ as well, since $A(t)$ is assumed to have multilinear rank $(r_1, \dots, r_d)$:
\begin{align*}
 \Y_{[i]}(t) &= \Y_{[i]}(t_0) + \underset{t_0}{\overset{t}{\int}} \dt{\Y}_{[i]}(s) ds \\
 &= \mat_i\Bigl(A(t_0) \underset{k=1}{\overset{i-1}{\bigtimes}} \U_k^{1,T} \Bigr) + \mat_i\Bigl(\bigl(A(t) - A(t_0)\bigr) \underset{k=1}{\overset{i-1}{\bigtimes}} \U_k^{1,T}\Bigr) \\
 &= \mat_i\Bigl(A(t) \underset{k=1}{\overset{i-1}{\bigtimes}} \U_k^{1,T} \Bigr) \\
 &= \U_i(t) \S_i(t) \V_i(t)^T, 
\end{align*}
where we use the decomposition \eqref{USW} and set 
\begin{align}
\V_i(t) &= (\U_1^{1} \otimes \cdots \otimes \U_{i-1}^{1} \otimes \I_i \otimes \cdots \otimes \I_d) \, \W_i(t)  \label{def_V}
\end{align}

To show \eqref{initYi}, we use an induction argument. With the abbreviation $\Delta A = A(t_1) - A(t_0)$ we have
\begin{align*}
  \U_{i-1}^1\L_{i-1}^{0,T} &= \U_{i-1}^1 \widehat{\S}_{i-1}^1 \V_{i-1}^{0,T} - \U_{i-1}^1 \U_{i-1}^{1,T} \mat_{i-1}(\Delta A) \V_{i-1}^0 \V_{i-1}^{0,T} \\
  &= \U_{i-1}^0 \S_{i-1}^0 \V_{i-1}^{0,T} + \mat_{i-1}(\Delta A) \V_{i-1}^0 \V_{i-1}^{0,T} \\
  &\qquad - \U_{i-1}^1 \U_{i-1}^{1,T} \mat_{i-1}(\Delta A) \V_{i-1}^0 \V_{i-1}^{0,T} \\
  &= \mat_{i-1} \Bigl(A(t_0) \underset{k=1}{\overset{i-2}{\bigtimes}} \U_{k}^{1,T} \Bigr) \\
  &\qquad + \bigl( \I - \U_{i-1}^1 \U_{i-1}^{1,T} \bigr) \bigl(\mat_{i-1}(\Delta A) \V_{i-1}^0 \V_{i-1}^{0,T}\bigr),
 \end{align*}
where the last equality holds by the induction hypothesis. It follows that 
$$ \L_{i-1}^{0,T} = \U_{i-1}^{1,T} \mat_{i-1} \Bigl(A(t_0) \underset{k=1}{\overset{i-2}{\bigtimes}} \U_{k}^{1,T} \Bigr) = \mat_{i-1}\bigl(A(t_0) \underset{k=1}{\overset{i-1}{\bigtimes}} \U_k^{1,T} \bigr).$$ Retensorizing and taking the $i$-mode unfolding yields $$\Y_{[i]}(t_0) = \mat_i(\ten_{i-1}(\L_{i-1}^{0,T})),$$ which becomes \eqref{initYi} with the above formula for $\L_{i-1}^{0,T}$.

To show the exactness of Alg.~\ref{algntucker}, we first consider the $d$-mode unfolded subproblem. Here, the last substep of the nested Tucker integrator is the same as applying the matrix projector-splitting integrator to \eqref{expl} with initial value \eqref{initYi} for $i=d$. Since the updated basis matrices $\U_k^1$ for $k = 1, \dots i-1$ are not time-dependent from the $i$-th integration step onwards, we observe by means of~\eqref{def_V}, that 
\begin{align*}
\V_i(t_1)^T \V_i(t_0) 
&= \W_i^T(t_0) \, (\U_1^{1,T} \U_1^{1} \otimes \cdots \otimes \U_{i-1}^{1,T} \U_{i-1}^{1} \otimes \I_i \otimes \cdots \otimes \I_d) \, \W_i(t_0) \\
&= \W_i(t_1)^T \W_i(t_0),
\end{align*}
for all $i = 1, \dots, d$. Additionally, by assumption $\W_i(t_1)^T\W_i(t_0)$ is non-singular and so we conclude by 
Theorem~\ref{matex} that the integrator is exact for the $d$-mode setting after one time step from $t_0$ to $t_1$: 
\begin{align*}
 \Y_{[d]}^1 = \mat_d\Bigl(A(t_1) \underset{k=1}{\overset{d-1}{\bigtimes}} \U_k^{1,T}\Bigr).
\end{align*}
We now show by induction for $i = d, \dots, 1$, that 
\begin{align}\label{solYi}
 \Y_{[i]}^1 = \mat_i\Bigl(A(t_1) \underset{k=1}{\overset{i-1}{\bigtimes}} \U_k^{1,T}\Bigr).
\end{align}
Suppose this has been shown for $\Y_{[d]}^1, \dots, \Y_{[i+1]}^1$. The substep of Alg.~\ref{algntucker} in the $i$-mode unfolding solves exactly the differential equations 
 \begin{align*}
  \dt{\K}_i(t) &= \mat_i\Bigl(\dt{A}(t) \underset{k=1}{\overset{i-1}{\bigtimes}} \U_k^{1,T}\Bigr) \V_i^0, \qquad \K_i(t_0) = \Y_{[i]}^0 \V_i^0 \\
  \dt{\S}_i(t) &= \U_i^{1,T} \mat_i\Bigl(\dt{A}(t) \underset{k=1}{\overset{i-1}{\bigtimes}} \U_k^{1,T}\Bigr) \V_i^0, \qquad \S_i(t_0) = \U_i^{1,T}\Y_{[i]}^0\V_i^0,
 \end{align*}
and approximately the differential equation 
\begin{align*}
 \dt{\L}_i^T(t) &= \U_i^{1,T} \mat_i\Bigl(\dt{A}(t) \underset{k=1}{\overset{i-1}{\bigtimes}} \U_k^{1,T}\Bigr), \qquad \L_i^T(t_0) = \U_i^{1,T} \Y_{[i]}^0.
\end{align*}
Since $\Y_{[i+1]}^1$ is the exact solution for the $(i+1)$-mode setting, we conclude by induction hypothesis 
\begin{align*}
 \L_i^{1,T} &= \mat_i\bigl(\ten_{i+1}(\Y_{[i+1]}^1)\bigr) = \mat_i\Bigl(A(t_1) \underset{k=1}{\overset{i}{\bigtimes}} \U_k^{1,T} \Bigr) \\
 &= \U_i^{1,T} \mat_i\Bigl(A(t_1) \underset{k=1}{\overset{i-1}{\bigtimes}} \U_k^{1,T} \Bigr) = \L_i^T(t_1).
\end{align*}
Hence also the differential equation in the third substep of the $i$-mode unfolded subproblem is solved exactly. By the exactness result for the matrix projector-splitting integrator, Alg.~\ref{algntucker} solves \eqref{expl} with initial value \eqref{initYi} exactly, so that \eqref{solYi} is satisfied. Hence, \eqref{solYi} holds also for $i=1$, which yields $Y^1 = A(t_1)$.
\end{proof}

\section{Error bounds for the nested Tucker integrator}\label{sec:error_robustness}

We now show that, just like in Thm.~\ref{thm:robustMatrix} for the matrix case, the nested Tucker integrator is robust to small singular values. Since this integrator is based on recursively applying the matrix projector-splitting integrator, the plan is to analyse these recursive steps from the matrix perspective so that we can  apply Thm.~\ref{thm:robustMatrix}. To this end, we first need to generalise the assumptions of Thm.~\ref{thm:robustMatrix}.

 Let $A(t)$ be the solution of~\eqref{ode} on $[t_0,T]$. We denote again by $\M$ the manifold of tensors of multilinear rank $(r_1, \dots, r_d)$. Let $$\M_i = \{Y \in \R^{n_1 \times \cdots \times n_d} \colon \rank(\mat_i(Y))=r_i\},$$ so that $\M = \M_1 \cap \dots \cap \M_d$. We assume that for each $i=1, \dots, d$, the $i$-mode unfolding of~\eqref{ode} satisfies the following conditions.
 \begin{itemize}
 \item $F(t,Y)$ is Lipschitz continuous and bounded for all $Y, \widetilde{Y} \in \R^{n_1 \times \cdots \times n_d}$:
 \begin{align}\label{Lip}
  \| F(t,Y) - F(t,\widetilde{Y}) \| \leq L \| Y - \widetilde{Y} \|, \qquad \| F(t,Y) \| \leq B.
 \end{align}
 \item $F(t,Y)$ can be decomposed into a tangential part and a small perturbation: 
 \begin{align}\label{MR}
 \begin{split}
  &F(t,Y) = M_i(t,Y) + R_i(t,Y), \\
  &M_i(t,Y) \in \T_Y\M_i, \quad \|R_i(t,Y)\| \leq \varepsilon,
 \end{split}
 \end{align}
 for all $Y\in \M_i$ in a neighborhood of $A(t)$ and for all $t \in [t_0,T]$. 
\item The initial value $A(t_0)$ for \eqref{ode} has multilinear rank $(r_1, \dots, r_d)$.
 \end{itemize}
 
 The second condition~\eqref{MR} is formulated in terms of $\M_i$ that are essentially fixed matrix manifolds. Since we are solving~\eqref{dlra} on a fixed rank Tucker manifold $\M$, it seems more natural to impose that $F(t,Y)$ is close to the tangent space of $\M$, that is,
 \begin{equation}\label{MR_M}
  \| F(t,Y) - P(Y) F(t,Y) \| \leq \varepsilon.
 \end{equation}
 However, since $\M = \M_1 \cap \dots \cap \M_d$, by definition of a tangent space we get $T_Y\M \subseteq T_Y \M_1 \cap \dots \cap T_Y \M_d$ for $Y \in \M$. Hence, $P(Y) F(t,Y) \in T_Y \M_i$ for all $i=1,\ldots,d$ and so~\eqref{MR_M} actually implies~\eqref{MR} for all $Y \in \M$.
 
 \begin{theorem}
 Under the above assumptions, the error of the nested Tucker integrator after $n$ steps with step size $h>0$ satisfies for all $t_n = t_0 + nh \leq T$: 
\begin{align*}
 \| Y_n - A(t_n) \| \leq c_1 h + c_2 \varepsilon,
\end{align*}
where the constants $c_1, c_2$ only depend on $L,B,T-t_0$ and the dimension $d$. In particular, the constants are independent of singular values of matricizations of the exact or approximate solution tensor.
\end{theorem}
\begin{proof}
Recall from~\eqref{odeYi} and~\eqref{Li} for the derivation of the nested Tucker integrator, that Alg.~\ref{alg:nestedKSL} solves approximately the following subproblem for each mode $i$ on $\R^{r_1 \times \cdots \times r_{i-1} \times n_i \times \cdots \times n_d}$:
\begin{align*}
 \dt{Y}_{i}(t) = F\Bigl(t,Y_{i}(t) \underset{k=1}{\overset{i-1}{\bigtimes}} \U_k^1 \Bigr)  \underset{k=1}{\overset{i-1}{\bigtimes}} \U_k^{1,T}, \qquad Y_{i}(t_0) = \ten_{i-1}(\L_{i-1}^{0,T})
 \end{align*}
 with $\L_{i-1}^{0,T} = \widetilde{\S}_{i-1}^0\V_{i-1}^{0,T}$. Introducing 
 $$Z_i(t) = Y_i(t) \underset{k=1}{\overset{i-1}{\bigtimes}} \U_k^1 \ \in \M_1 \cap \dots \cap \M_{i-1} \subset \R^{n_1 \times \cdots \times n_d},$$ we obtain the equivalent initial value problem on $\R^{n_1 \times \cdots \times n_d}$
\begin{align}\label{odeZi}
\dt{Z}_{i}(t) = F\bigl(t,Z_{i}(t)\bigr) \underset{k=1}{\overset{i-1}{\bigtimes}} (\U_k^1\U_k^{1,T}), \qquad
 Z_{i}(t_0) = \ten_{i-1}(\L_{i-1}^{0,T})\underset{k=1}{\overset{i-1}{\bigtimes}} \U_k^1.
 \end{align}
 We note that since $\L_{i-1}^{0,T}$ has full rank, we have $Z_i(t_0) \in \M_i$. Alg.~\ref{alg:nestedKSL} now applies the matrix projector-splitting integrator with inexact integration in the third substep to the $i$-mode unfolded differential equation~\eqref{odeZi}. This results in the approximation $Z_i^1 \in \M_i$ to $Z_i(t_1)$. 
 
 We show by induction for $i = d, \dots, 1$ the local error bound 
\begin{align}\label{locerr}
 \| Z_i^1 - Z_i(t_1) \| = \bigO(h(\varepsilon + h)),
\end{align}
where the constants symbolized by the $\bigO$ notation depend only on $L,B$ and $d$. For $i=d$, the approximation is obtained by the matrix projector-splitting algorithm with exact solution of all three substeps. We verify that the conditions~(a--c) of Thm.~\ref{thm:robustMatrix} applied to~\eqref{odeZi} are satisfied: 
Assumption~(a) is trivially satisfied by~\eqref{Lip} since $\U_k^1\U_k^{1,T}$ is an orthogonal projector. Using~\eqref{MR}, the $d$-mode unfolding of the right hand side of \eqref{odeZi} can be decomposed, for $Y \in \M_d$, as 
\begin{align*}
 \mat_d \Bigl(F(t,Y) \underset{k=1}{\overset{d-1}{\bigtimes}} (\U_k^1\U_k^{1,T})\Bigr) =& \mat_d \Bigl(M_d(t,Y) \underset{k=1}{\overset{d-1}{\bigtimes}} (\U_k^1\U_k^{1,T})\Bigr) \\ 
 &+ \mat_d \Bigl(R_d(t,Y) \underset{k=1}{\overset{d-1}{\bigtimes}} (\U_k^1\U_k^{1,T})\Bigr), 
\end{align*}
where $M_d(t,Y) \in \T_Y\M_d$ and $\| R_d(t,Y) \| \leq \varepsilon$. 

We note that if $M_d(t,Y) \in \T_Y\M_d$ and $Y=Y \bigtimes_{k=1}^{d-1} (\U_k^1\U_k^{1,T})$ (as is the case for $Y=Z_i(t)$ in \eqref{odeZi}), then we also have $M_d(t,Y) \bigtimes_{k=1}^{d-1} (\U_k^1\U_k^{1,T}) \in \T_Y\M_d$. 
This holds true because if we consider the singular value decomposition of $\mat_d(Y) =\U\S\V^T$, then $\V^T=\V^T \bigotimes_{k=1}^{d-1} \U_k^1\U_k^{1,T}$. Now, $M \in \T_Y\M_d$ means that $\mat_d(M)=\delta\!\U\S\V^T + \U\delta\S\V^T + \U\S\,\delta\!\V^T$ for some suitable $\delta\!\U,\delta\S,\delta\!\V$. But then, since  $\V^T=\V^T \bigotimes_{k=1}^{d-1} \U_k^1\U_k^{1,T}$, this implies that
$\mat_d\bigl(M \bigtimes_{k=1}^{d-1} (\U_k^1\U_k^{1,T})\bigr) $ is of the same form with a modifed $\delta\!\V$, and hence $M \bigtimes_{k=1}^{d-1} (\U_k^1\U_k^{1,T})\in \T_Y\M_d$.

By definition, $\mat_d(\M_d) = \{\mat_d(Y) \, \colon \, Y \in \M_d\}$ is the manifold of matrices of rank $r_d$ of dimension $(n_d \times n_1 \cdots n_{d-1})$. Moreover, for $Y \in \M_d$ and $\Y = \mat_d(Y)$, we have $\T_{\Y}\mat_d(\M_d) = \mat_d(\T_Y\M_d)$. We conclude that 
\begin{align*}
 \mat_d\Bigl(M_d(t,Y) \underset{k=1}{\overset{d-1}{\bigtimes}} (\U_k^1\U_k^{1,T})\Bigr) \in \T_{\Y}\mat_d(\M_d)
\end{align*}
and the corresponding term with $R_d$ is still bounded by $\varepsilon$ thanks to~\eqref{MR}. Hence, assumption~(b) is verified. Since assumption~(c) was shown above, we are now in the situation to apply Thm.~\ref{thm:robustMatrix}, which yields \eqref{locerr} for $i=d$. 

We proceed similarly for $i=d-1$ down to $1$. In these cases, we apply the matrix projector-splitting algorithm to the $i$th unfolding with an inexact solution of the third substep. The error of this inexact solution is given by \eqref{locerr} for $i+1$. In the same way as before, the conditions of Thm.~\ref{thm:robustMatrix} are verified for the rank $r_i$ matrix manifold $\mat_i(\M_i)$. With the induction hypothesis that \eqref{locerr} holds for $i+1, \dots, d$, we conclude from
the error bound \eqref{err-inexact} (for
the situation of inexact solutions in the substeps) that \eqref{locerr} also holds for $i$.  

For $i=1$, this gives the stated error bound.
\end{proof}

\section{Equivalence with the tensor projector-splitting integrator of \cite{L15}}
The nested Tucker integrator presented in the previous section goes through each mode and reduces the dimension of the current mode before updating the basis matrix $\U_i(t)$ and the matrix $\S_i(t)$ for all $i = 1, \dots, d$. 

In contrast to this, the time integrator described in~\cite{L15} does not reduce the dimension in each mode, since the corange of the current unfolded approximation tensor takes the basis matrices after performing one time step. This step of the integrator and the full algorithm can be retraced in the following algorithm: \\
\newline
\begin{algorithm}[H]\label{algLubichtucker}
\caption{One time step of the Tucker integrator}
    \label{algtucker}
        \KwData{Tucker tensor $Y^0 = C^0 \bigtimes_{i=1}^d \U_i^0$, $F(t,Y)$, $t_0$, $t_1$}
        \setcounter{AlgoLine}{0}
        \KwResult{Tucker tensor $Y^1 = C^1 \bigtimes_{i=1}^d \U_i^1$}
    \Begin{
           \For {$i=1$ \KwTo $d$}
           {
             compute QR factorization $\mat_i(C^0)^T = \Q_i^0 \S_i^{0,T}$ 
            
             set $\V_i^{0,T} = \mat_i \Bigl(\ten_i(\Q_i^{0,T}) \underset{k=1}{\overset{i-1}{\bigtimes}} \U_k^{1,T} \underset{l=i+1}{\overset{d}{\bigtimes}} \U_l^{0,T}\Bigr)$ \label{alg:line:LubichTucker-Vi}
            
             set $\K_i^0 = \U_i^0 \S_i^0$
             
             set $\Y_{[i]}^+(t) = \K_i(t)\V_i^{0,T}$
            
             solve $\dt{\K}_i(t) = \mat_i\bigl(F(t, \ten_i(\Y_{[i]}^+))\bigr) \V_i^0$, \\ 
             \ with initial value $\K_i(t_0) = \K_i^0$ and return $\K_i^1 = \K_i(t_1)$ \label{alg:line:LubichTucker-dK}
            
             compute QR factorization $\K_i^1 = \U_i^1 \widehat{\S}_i^1$
            
             set $\S_i^0 = \widehat{\S}_i^1$
             
             set $\Y_{[i]}^-(t) = \U_i^1\S_i(t)\V_i^{0,T}$
            
             solve $\dt{\S}_i(t) = -\U_i^{1,T} \mat_i\bigl(F(t, \ten_i(\Y_{[i]}^-))\bigr) \V_i^0$, \\
             \ with initial value $\S_i(t_0) = \S_i^0$ and return $\S_i^1 = \S_i(t_1)$
            
             set $C^0 = \ten_i(\S_i^1 \Q_i^{0,T})$
           }
        
           solve $\dt{C}(t) = F\bigl(t, C(t) \underset{i=1}{\overset{d}{\bigtimes}} \U_i^1\bigr)\underset{i=1}{\overset{d}{\bigtimes}} \U_i^{1,T}$, \\
           \ with initial value $C(t_0) = C^0$ and return $C^1 = C(t_1)$
        
           set $Y^1 = C^1 \underset{i=1}{\overset{d}{\bigtimes}} \U_i^1$ 
         }
\end{algorithm}
\vspace{10pt}
Comparing Alg.~\ref{algLubichtucker} for the Tucker integrator with Alg.~\ref{alg:nestedKSL} for the nested Tucker integrator, we see obvious similarities in the structure of those two methods. However, the two algorithms solve different matrix differential equations for $\K_i(t)$ and $\S_i(t)$---clearly they have different definitions but they also differ in size. In addition, one algorithm integrates $C(t)$ whereas the other $\L(t)$. 
%
This distinction of equations is due to different definitions for the coranges of the unfoldings of the approximation tensor $Y$. Nevertheless, we will see in the following that both integration methods are equivalent. 
\begin{theorem}
 The nested Tucker integrator presented in Section~\ref{NTucker} (Alg.~\ref{algntucker}) and the Tucker integrator described in~\cite{L15} (Alg.~\ref{algLubichtucker}) applied on the tensor differential equation \eqref{ode} are equivalent in the sense that they yield the same low-rank approximation $Y^1$ after one time step. 
\end{theorem}
\begin{proof}
It is sufficient to show equivalence of the matrix differential equations that appear in Alg.~\ref{algntucker} and~\ref{algtucker}. In order to distinguish between the factors computed by those two methods, we will denote those for Alg.~\ref{algtucker} using $\overline{\ \cdot\ }$, e.g., $\overline{\V}_i$, $\overline{\K}_i(t)$, and $\overline{\S}_i(t)$. The notation for Alg.~\ref{algntucker} is left unchanged.


We start with Alg.~\ref{algtucker}. Writing line~\ref{alg:line:LubichTucker-Vi} as
\begin{align*}
 \overline{\V}_i^{0,T} 
 &=\Q_{i}^{0,T}  \Bigl(\Bigotimes_{k=1}^{i-1} \U_{k}^{1,T} \otimes\Bigotimes_{k=i+1}^d \U_k^{0,T} \Bigr),
\end{align*}
the equation of motion of $\K_i$ becomes
\begin{align*}
\dt{\overline{\K}}_i(t) &= \mat_i\Bigl(F\bigl(t,\ten_i(\overline{\K}_i(t)\overline{\V}_i^{0,T})\bigr)\Bigr) \overline{\V}_i^0 \\
&= \mat_i\Bigl(F\Bigl(t,\ten_i\Bigl(\overline{\K}_i(t)  \Q_{i}^{0,T}  \Bigl(\Bigotimes_{k=1}^{i-1} \U_{k}^{1,T} \otimes \Bigotimes_{k=i+1}^d \U_k^{0,T} \Bigr)  \Bigr)\Bigr) \overline{\V}_i^0.
\end{align*}
For Alg.~\ref{algntucker}, on the other hand, that equation  reads
\[
\dt{\K}_i(t) = \mat_i\Bigl(F\Bigl(t,\ten_i(\K_i(t) \V_i^{0,T}) \underset{k=1}{\overset{i-1}{\bigtimes}} \U_k^1 \Bigr) \underset{k=1}{\overset{i-1}{\bigtimes}} \U_k^{1,T} \Bigr) \V_i^{0}.
\]
We first expand the argument of $F$ in this ODE. Writing line~\ref{alg:line:V_i} in Alg.~\ref{algntucker} as
\begin{align*}
 \V_i^{0,T} = 
 \Q_{i}^{0,T}  \Bigl(\Bigotimes_{k=1}^{i-1} \I_{r_k} \otimes\Bigotimes_{k=i+1}^d \U_k^{0,T} \Bigr)
\end{align*}
and substituting, we obtain
\begin{align*}
\ten_i(\K_i(t) \V_i^{0,T}) \underset{k=1}{\overset{i-1}{\bigtimes}} \U_k^1 &=
\ten_i(\K_i(t) \Q_{i}^{0,T} ) \underset{k=i+1}{\overset{d}{\bigtimes}} \U_k^{0} \underset{k=1}{\overset{i-1}{\bigtimes}} \U_k^1 \\
&= \ten_i\Bigl(\K_i(t) \Q_{i}^{0,T}  \Bigl( \Bigotimes_{k=i+1}^d \U_k^{0,T} \otimes\Bigotimes_{k=1}^{i-1} \U_{k}^{1,T} \Bigr)  \Bigr).
\end{align*}
Hence, we see that $F$ has the same arguments in both algorithms for the K-step. Omitting it, we continue with
\begin{align*}
\dt{\K}_i(t) &= \mat_i\Bigl(F(t, \cdot ) \underset{k=1}{\overset{i-1}{\bigtimes}} \U_k^{1,T} \Bigr) \V_i^{0} \\
&=\mat_i\Bigl(F(t, \cdot ) \Bigr)  \Bigl( \underset{k=1}{\overset{i-1}{\Bigotimes}} \U_k^{1,T} \otimes\Bigotimes_{k=i+1}^{d} \I_{r_k}  \Bigr) \Bigl(\Bigotimes_{k=1}^{i-1} \I_{r_k} \otimes\Bigotimes_{k=i+1}^d \U_k^{0} \Bigr) \Q_{i}^{0}  \\
&=\mat_i\Bigl(F(t, \cdot ) \Bigr)  \Bigl( \underset{k=1}{\overset{i-1}{\Bigotimes}} \U_k^{1,T}  \otimes\Bigotimes_{k=i+1}^d \U_k^{0} \Bigr) \Q_{i}^{0}.  
\end{align*}
Comparing with $\overline{\V}_i^{0}$ above, we see that the differential equations for $\K_i(t)$ and $\overline{\K}_i(t)$ are indeed equivalent.

Hence, applying the same numerical method to both of them would give the same result $\K_i^1 = \overline{\K}_i^1$. 

The equivalence of the evolution equations for $\dt{\S}_i(t)$ and $\dt{\overline{\S}}_i(t)$ can be shown in the same way as above. 
This again  gives the same numerical solutions after one time step, i.e., $\S_i^1 = \overline{\S}_i^1$. 

Finally, for the core tensor, we compare the evolution equation for for $C(t)$ in Alg.~\ref{algtucker},
\begin{align*}
 \dt{C}(t) = F\bigl(t, C(t) \underset{i=1}{\overset{d}{\bigtimes}} \U_i^1\bigr)\underset{i=1}{\overset{d}{\bigtimes}} \U_i^{1,T}, \qquad C(t_0) = C^0,
\end{align*}
with that of $\L(t)$ from Alg.~\ref{algntucker}. Rerensorizing the latter in the $d$th mode yields
\begin{align*}
 \ten_d(\dt{\L}^T(t)) &= \ten_d\Bigl(\U_d^{1,T} \mat_d\Bigl(F\Bigl(t,\ten_d(\U_d^1   L^T(t)) \underset{i=1}{\overset{d-1}{\bigtimes}} \U_i^1 \Bigr) \underset{i=1} {\overset{d-1}{\bigtimes}} \U_i^{1,T}\Bigr)\Bigr) \\
                       &= \ten_d\Bigl(\U_d^{1,T} \mat_d\Bigl(F\Bigl(t,\ten_d(\L^T(t)) \underset{i=1}{\overset{d}{\bigtimes}} \U_i^1 \Bigr) \underset{i=1}{\overset{d-1}{\bigtimes}} \U_i^{1,T}\Bigr)\Bigr) \\
                       &= F\bigl(t, \ten_d(\L^T(t)) \underset{i=1}{\overset{d}{\bigtimes}} \U_i^1 \bigr) \underset{i=1}{\overset{d}{\bigtimes}} \U_i^{1,T}.
\end{align*}
Identifying now $C(t)$ as $\ten_d({\L}^T(t))$, we see that differential equations are the same. Since the same holds true for the initial values,
\begin{align*}
 \ten_d(\L_d^{0,T}) = \ten_d(\mat_d(C^0)) = C^0,
\end{align*}
both algorithms deliver the same same low-rank approximation $Y^1$.
\end{proof}

\section{Numerical experiments}
We present two numerical examples to illustrate our theoretical results of the proposed Tucker integrator. We consider examples that are tensor variants of the examples in~\cite[Section 4]{KLW16} for the matrix case, in particular, a discrete nonlinear Schr\"odinger equation and an example of approximately adding tensors in the Tucker format. 

\subsection{A discrete nonlinear Schr\"odinger equation}
We model a dilute Bose--Einstein condensate, trapped in a periodic potential~\cite{ST01}, on a regular lattice of width $\gamma$. The dynamics of its phase diagram is governed by the discrete nonlinear Schr\"odinger equation
\begin{align}\label{DNLS}
\begin{split} 
i \dt{A}(t) &= -\frac{1}{2}L[A(t)] + \varepsilon |A(t)|^2 \odot A(t) \\
A_{jkl}(t_0) &= \exp\bigl(-1/\gamma^2((j-j_1)^2 - (k-k_1)^2 - (l-l_1)^2)\bigr) \\
&\quad + \exp\bigl(-1/\gamma^2((j-j_2)^2 - (k-k_2)^2 - (l-l_2)^2)\bigr), 
\end{split}
\end{align}
where $A(t) \in \R^{n_1 \times n_2 \times n_3}$ with $n_i = 100$ for all $i \in \{1, 2, 3\}$. The bounded linear operator $L\colon \R^{n_1 \times n_2 \times n_3} \to \R^{n_1 \times n_2 \times n_3}$ describes the interaction between the grid points centred at $(j,k,l)$ for all $j,k,l = 1, \dots, 100$. It is defined component-wise as 
\begin{align*}
L[A](j,k,l) =&\ A(j-1,k,l)+ A(j+1,k,l) + A(j,k-1,l)+A(j,k+1,l)
\\
&+A(j,k,l-1)+A(j,k,l+1),
\end{align*}
where terms with indices outside the range from 1 to 100 are interpreted as~$0$.
Compared to the more standard seven-point stencil for the discrete Laplace operator in three dimensions, the operator $L$ does not take the centred grid point into account. The entries of the tensor $|A|^2$ are the squares of the absolute values of the corresponding entries of $A$, and the $\odot$ in $|A|^2 \odot A$ desnotes the entrywise (Hadamard) product. The parameter $\varepsilon$ determines the degree of nonlinearity.

We consider two excitations of the system, which are located at grid-points $(j_1,k_1,l_1) = (75,25,1)$ and $(j_2,k_2,l_2) = (25,75,100)$. We take $\gamma = 10$. 

To compute a low-rank approximation $Y(t) \in \M$ with multilinear rank $r=(10,10,10)$ to the solution of the nonlinear differential equation \eqref{DNLS}, we apply our nested Tucker integrator (Alg.~\ref{alg:nestedKSL}) to~\eqref{DNLS}. The differential equations appearing in the substeps of each mode are solved by the $4$th order Runge--Kutta method with time step size $h = 10^{-3}$. This approximate solution is compared to a full rank reference solution, which is also computed by a $4$th order Runge--Kutta method, but with $h = 0.5 \cdot 10^{-3}$. In Table \ref{err-table} we show the error behavior for different parameters $\varepsilon$ and time step sizes $h$:

\begin{table}[htb]
	\centering
	\begin{tabular}{c|cccc}
		$\varepsilon$ \textbackslash\ $h$ & $1$ & $10^{-1}$ & $10^{-2}$ & $10^{-3}$ \\
		\hline
		$1$        & 4.59e-1  & 4.01e-2  & 3.88e-2  & 3.88e-2 \\
		$10^{-1}$  & 9.39e-2  & 9.68e-4  & 1.61e-4  & 1.47e-4 \\
		$10^{-2}$  & 9.27e-3  & 3.20e-5  & 2.19e-6  & 1.30e-6 \\
		$10^{-3}$  & 5.36e-4  & 3.18e-6  & 8.93e-8  & 3.54e-8 \\
		$10^{-4}$  & 5.12e-5  & 2.73e-7  & 3.23e-9  & 1.91e-9
	\end{tabular}
\caption{Error in Frobenius norm at $t=1$ of the rank-(10,10,10) Tucker integrator applied to \eqref{DNLS}. \label{err-table}}
\end{table}
For each time step size $h$, we see the error decaying with $\varepsilon$. This observation is due to the fact that the linear term $L[A(t)]$ in \eqref{DNLS} maps onto the tangent-space $\TM$ of the manifold $\M$ of multilinear rank. The nonlinear term is of full rank, but it is controlled by the factor $\varepsilon$. This makes the dependence of the error behaviour on $\varepsilon$ explicit. 
We also see in the first row, that the error stagnates from time step size $h=10^{-2}$ on and this shows the dominance of the perturbation factor $\varepsilon$. We would observe the same behaviour for smaller $\varepsilon$, but for smaller time step sizes. 

Finally, in the last row, where the influence of $\varepsilon$ is small, we observe convergence of the error in terms of the time step size $h$.

\subsection{Approximate addition of tensors}
Let $A \in \mathbb{C}^{n_1 \times \cdots \times n_d}$ be a tensor of multilinear rank $r=(r_1, \dots, r_d)$ and let $B \in \mathbb{C}^{n_1 \times \cdots \times n_d}$. We consider the addition of the two given tensors, which results in 
\begin{align}\label{diradd}
 C = A + B,
\end{align}
where $C$ typically is not of low rank. We aim to find an approximation tensor of multilinear rank $r$. Such a computation is for example required in optimization problems on low-rank manifolds, under the name of retractions, and need to be computed in each iterative step~\cite{AO14}. There,  the increment is typically a tangential tensor $B \in \mathcal{T}_A\mathcal{M}$, which after adding directly as in \eqref{diradd} yields a tensor $C$ of multilinear rank 2r. Afterwards, the result is projected onto $\mathcal{M}$ by an SVD-based rank-$r$ approximation in order to obtain an approximation tensor $D \in \M$. With this procedure,  we first leave the low-rank manifold  and then project back onto $\M$.

Instead, we propose to apply one time step of the nested Tucker integrator starting with $t_0=0$ and with time step size $h = 1$ in order to solve 
\begin{align*}
 \dt{Y}(t) = P(Y)B, \qquad Y(t_0) = A.
\end{align*}
This gives an approximate solution $Y^1 \in \M$ for the result of the direct addition \eqref{diradd}. Contrary to the standard approach, we never leave the low-rank manifold when applying the nested Tucker integrator.

For our numerical example, we initialise $A$ as a random Tucker tensor of size $100 \times 100 \times 100$ and multi-linear rank $r = (10,10,10)$. The increment $B$ is constructed to be a random tensor in the tangent space $\mathcal{T}_Y\mathcal{M}$. We compare the full rank addition \eqref{diradd} with the low-rank approximation $Y^1 \in \M$ obtained by the nested Tucker integrator. We also compare those results with the retracted rank-$2r$ result, for which we perform a best rank-$r$ approximation. The figure below illustrates those comparisons: 

\begin{figure}[h!]
	\centering
	\includegraphics[trim = 7cm 9cm 7cm 9cm, width=0.25\textwidth]{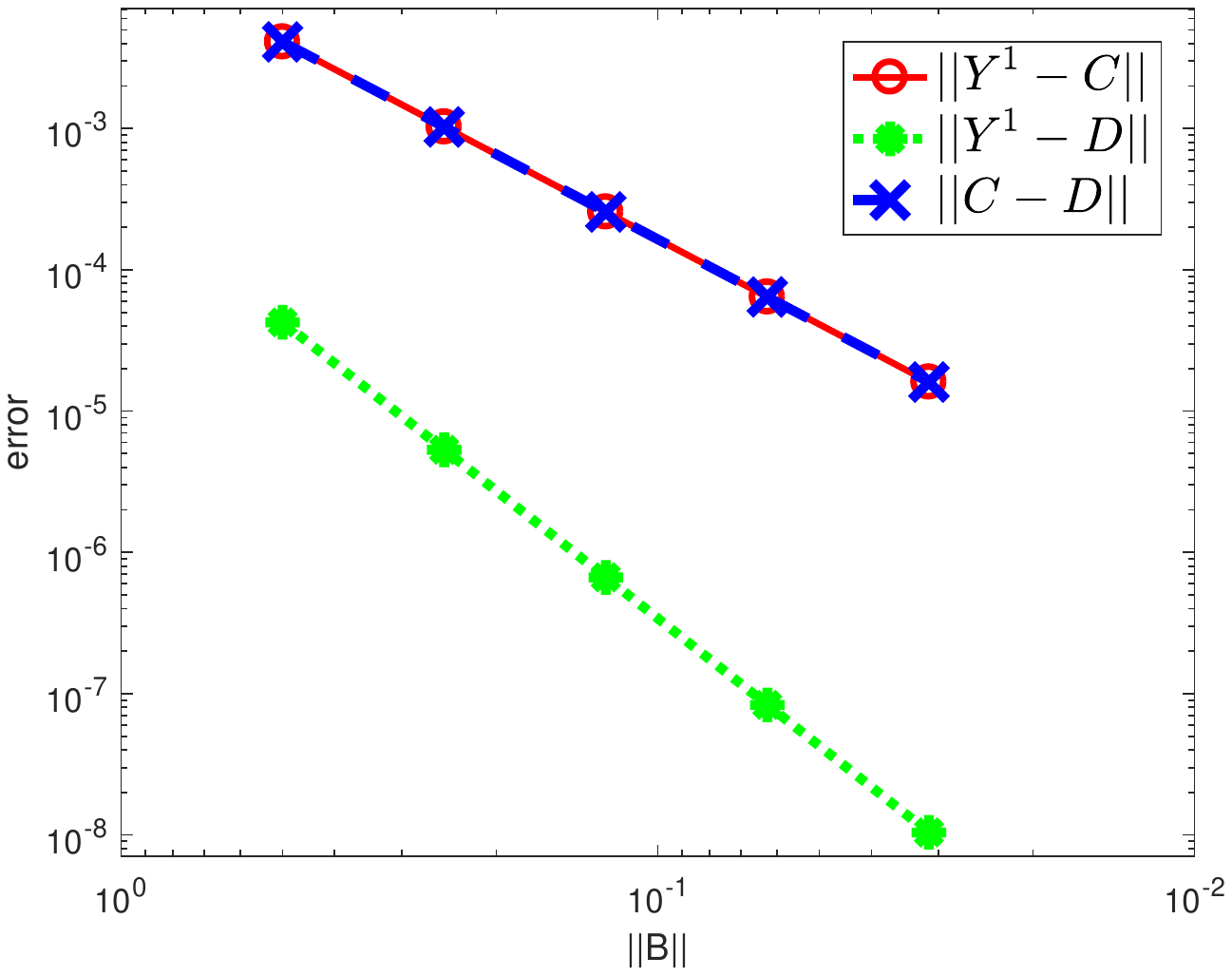}
	\caption{Errors for tensor addition for tangential increments $B$ of decreasing norm.}
	\label{fig:addition}
\end{figure}
We observe that the errors decrease with decreasing norm of the increment tensor $B$. We also see that the difference between the errors of the splitting integrator and the projected direct addition is marginal --- or rather we do not see it because the error curves of both approaches are not distinguishable in the figure.

\section*{Acknowledgements}
This work was supported by a grant from DFG through the GRK 1838. We thank Balázs Kovács for helpful discussions about numerical examples.


\end{document}